\documentclass[11pt]{amsart}
\usepackage{amsmath,amsthm,amsfonts,amscd,amssymb,eucal,latexsym,mathrsfs,stmaryrd,enumitem,bm,microtype}
\usepackage[all]{xy}
\usepackage{mathtools}
\usepackage{enumitem}  
\usepackage{xcolor}
\usepackage{comment}

\newtheorem*{theorem*}{Theorem}
\newtheorem{theorem}{Theorem}[section]
\newtheorem{corollary}[theorem]{Corollary}
\newtheorem{lemma}[theorem]{Lemma}

\theoremstyle{definition}
\newtheorem{definition}[theorem]{Definition}
\newtheorem{remark}[theorem]{Remark}

\newtheorem{example}[theorem]{Example}

\theoremstyle{plain}
\newcounter{theoremintro}
\newcounter{conjectureintro}
\newcounter{corollaryintro}
\newtheorem{theoremi}[theoremintro]{Theorem}

\newtheorem{conjecture}[conjectureintro]{Conjecture}

\newcommand{\cZ}{{\mathcal Z}}
\newcommand{\cS}{{\mathcal S}}
\newcommand{\cU}{{\mathcal U}}

\newcommand{\Nb}{{\mathbb N}}

\newcommand{\eps}{\varepsilon}
\newcommand{\abs}[1]{\left|#1\right|}
\DeclarePairedDelimiter\ceil{\lceil}{\rceil}
\DeclarePairedDelimiter\floor{\lfloor}{\rfloor}

\newcommand{\emb}{\hookrightarrow}
\newcommand{\act}[1][]{\overset{#1}{\curvearrowright}}
\newcommand{\Alt}{\mathrm{A}}

\newcommand{\Full}{\mathrm{F}}

\newcommand{\finsub}{\subset\!\subset}
\DeclareMathOperator{\Fix}{Fix}

\numberwithin{equation}{section}

\allowdisplaybreaks

\begin{document}

\title{Almost finiteness and groups of dynamical origin}

\author{Petr Naryshkin}
\address{Petr Naryshkin,
Alfréd Rényi Institute of Mathematics, Budapest, Reáltanoda utca 13-15, 1053, Hungary}
\email{pnaryshkin@renyi.hu}

\author{Spyridon Petrakos}
\address{Spyridon Petrakos,
Mathematisches Institut,
WWU M{\"u}nster, 
Einsteinstr.\ 62, 
48149 M{\"u}nster, Germany}
\email{spetrako@uni-muenster.de}

\begin{abstract}
We introduce the property of having good subgroups for actions of countable discrete groups on compact metrizable spaces, and show that it implies comparison when the acting group is amenable. As a consequence, free actions on finite-dimensional spaces of many notable amenable groups of dynamical origin are almost finite. For instance, this applies to topological full groups of Cantor minimal systems and the Basilica group. In particular, minimal such actions give rise to classifiable crossed products.
\end{abstract}

\date{\today}

\maketitle

\section{Introduction}
With the Elliott classification programme coming to a conclusion after about three decades of work by many people, establishing that simple, separable, unital, nuclear, $\cZ$-stable C$^*$-algebras in the UCT class are classified by the Elliott invariant, a lot of work has been done on determining which C$^*$-algebras satisfy these conditions (that is, which C$^*$-algebras are \emph{classifiable}). One particular natural class of examples that has drawn considerable attention in the last decade is that of crossed products arising from amenable topologically free minimal actions of countably infinite discrete groups on compact metrizable spaces. These assumptions on the action are equivalent to the crossed product C$^*$-algebra being separable, unital, nuclear, and simple, while they also automatically guarantee that it satisfies the UCT \cite{Tu99}. That only leaves the question of $\cZ$-stability.

When the acting group itself is amenable and the action is free, the modern approach to this problem relies on the property of almost finiteness (as well as the closely related notion of dynamical comparison). This property was originally introduced by Matui \cite{Mat12} for groupoids with a totally disconnected unit space. Later, Kerr \cite{Ker20} adapted it to the actions of amenable groups on general compact spaces and showed that it serves as a criterion for $\cZ$-stability of the associated crossed product. The central conjecture in the theory is that almost finiteness is automatic in the finite-dimensional case\footnote{There are actions of amenable groups on spaces with infinite covering dimension that do not have $\cZ$-stable crossed products \cite{GioKer10}, and therefore are not almost finite.}:

\begin{conjecture}
\label{conj: amen}
    All free actions of countably infinite amenable discrete groups on finite-dimensional compact metrizable spaces are almost finite.
\end{conjecture} 

After a series of advances \cite{DowZha19,KerSza20, ConJacMarSewTuc20, KerNar21, Nar22, Nar23}, the conjecture is now known to hold for the smallest class $\cS$ of infinite amenable groups such that
\begin{itemize}
    \item groups of subexponential growth belong to $\cS$,
    \item $\cS$ is closed under taking direct limits,
    \item if $H \in \cS$ and $H \lhd \Gamma$ then $\Gamma \in \cS$,
    \item if $\Gamma$ has finite normal subgroups of arbitrarily large cardinality then $\Gamma \in \cS$.
\end{itemize}

It is not easy to find amenable groups that do not belong to this class. For instance, all infinite elementary amenable groups are contained in $\cS$, and the famous question of Day from the 1950s \cite{Day57}, asking whether every amenable group is elementary amenable, was already difficult enough. It was answered negatively in 1985 in the celebrated work of Grigorchuk \cite{Gri85}, where he proved that the group he constructed in \cite{Gri80} has subexponential growth and is thus non-elementary amenable. Of course, this particular group still lies in the class $\cS$. However, his methods inspired a number of constructions and the development of the theory of automata groups. The latter, alongside groups whose origin lies in ergodic theory but have more recently drawn considerable interest from the point of view of geometric group theory and operator algebras, can collectively be referred to as \emph{groups of dynamical origin}. To the best of our knowledge, this is the only source of examples of non-elementary amenable groups.

We mention here two particular such examples. One is the topological full group construction, which was used to produce uncountably many pairwise non-isomorphic finitely generated simple amenable groups of exponential growth, through the celebrated work of Juschenko and Monod \cite{JusMon13} in conjunction with previous work of Matui \cite{Mat06, Mat13}; see \cite{dCor14} for a survey of the topic. This combination of properties ensures that these groups are not in the class $\cS$. The other is the Basilica group \cite{GriZuk02} and its generalizations. Although it is not entirely clear whether they belong to $\cS$, these groups are known not to be elementary subexponentially amenable \cite{BarVir05,DDF-ANT22}, that is, they cannot be constructed from groups of subexponential growth using elementary operations.

The initial motivation for the present paper was the question of whether the examples above satisfy Conjecture \ref{conj: amen}. As it turns out, the answer is positive, and, in fact, the techniques are applicable in a quite general setting. More precisely, the main result is as follows. We refer to Section \ref{sec: prelim} for the relevant terminology.

\begin{theoremi}
\label{thmIntro: amen}
Suppose that $\Gamma$ is an amenable group that admits a faithful micro-supported (Definition \ref{def:microsup}) action on a Hausdorff Baire space with no isolated points\footnote{this assumption can be slightly weakened, see Remark \ref{rem: microsup weakening}}. Then any free action of $\Gamma$ on a finite-dimensional space $X$ is almost finite.
\end{theoremi}

This result is widely applicable and confirms Conjecture \ref{conj: amen} for \emph{amenable} groups in the following two classes (see Subsection \ref{ssec: microsup} for precise conditions):
\begin{itemize}
    \item many topological full groups of \'{e}tale groupoids and their notable subgroups, including those of Cantor minimal systems.
    \item all weakly branch groups, including the Basilica-type groups.
\end{itemize}
We remark that many other prominent groups of dynamical origin can be written as one of the above, including Thompson-like groups, the IET group, and certain iterated monodromy groups. Combined with the prior results, this covers a vast class of examples of discrete amenable groups.

We obtain the theorem above by sharpening the methods pioneered in \cite{Nar23}. More precisely, we introduce a technical condition of having \emph{good subgroups} (Definition \ref{def:good_Subgroups}) as a weakening of having a normal subgroup for which the restricted action is almost finite. For actions of amenable groups, we obtain the corresponding generalization of the main result in \cite{Nar23} --- if an action has good subgroups then it has comparison (Theorem \ref{thm:amen_comp_from_g_s}). This condition turns out to be automatic for groups that admit micro-supported actions (Lemma \ref{lem:MicroSupAbelianEmbed}).

The paper is structured as follows. In Section \ref{sec: prelim} we recall the basics on almost finiteness, comparison, and micro-supported actions. In Section \ref{sec: Comp_from_subgrps} we define having good subgroups and show how it can be used to deduce dynamical subequivalence, leading to the main results of the paper.

\medskip

\noindent{\it Acknowledgements.}
The authors are grateful to Julian Kranz for many helpful comments and suggestions. They are also deeply thankful to the anonymous reviewer for pointing them towards micro-supported actions, as well as numerous other suggested improvements. The research was partially funded by the Deutsche Forschungsgemeinschaft (DFG, German Research Foundation) under Germany’s Excellence Strategy – EXC 2044 – 390685587, Mathematics Münster – Dynamics – Geometry – Structure; the Deutsche Forschungsgemeinschaft (DFG, German Research Foundation) – Project-ID 427320536 – SFB 1442, and ERC Advanced Grant 834267 - AMAREC. It was also partially funded by Dynasnet European Research Council Synergy project -- grant number ERC-2018-SYG 810115.

\section{Preliminaries}
\label{sec: prelim}

Unless otherwise stated, throughout this paper $\Gamma\act[\alpha] X$ will be an action by homeomorphisms of a countably infinite discrete group on a compact metrizable space. We will denote by $M(X)$ (resp. $M^{\Gamma}(X)$) the set of all regular Borel (resp. $\Gamma$-invariant regular Borel) probability measures on $X$. The symbol $\finsub$ will be used to denote finite subsets.

\subsection{Amenability, almost finiteness, and comparison}

Let us now introduce the main tools used in proving classifiability of crossed products.

\begin{definition}
    Let $A,B$ be two subsets of $X$. We say that $A$ is \emph{(dynamically) subequivalent} to $B$ if for any closed set $C\subseteq A$ there exists a finite open cover $\mathcal{U}$ of $C$ and elements $\{s_U\}_{U\in\mathcal{U}}$ in $\Gamma$ such that the sets $\{s_UU\}_{U\in\mathcal{U}}$ are pairwise disjoint and contained in $B$. We denote that by $A\precsim_\Gamma B$ (or simply $A\precsim B$ whenever there is no ambiguity).
\end{definition}

\begin{definition}
    We define the \emph{type semigroup} of the action to be the abelian semigroup generated by symbols $\{[U]:U\subseteq X\text{ open}\}$ subject to the relations $[U\sqcup V]=[U]+[V]$ and $[U]=[sU]$ for all $s\in\Gamma$.
\end{definition}

One can check that $\precsim_\Gamma$ for open subsets extends to a preorder (i.e. a binary relation that is reflexive and transitive) on the type semigroup (\cite[Lemma 2.2]{Ma21}). We will keep the same notation for this extension. Additionally, we will frequently identify an open set $U$ with the corresponding element $[U]$ of the type semigroup.

Note that Ma \cite{Ma21} has defined and studied the \emph{generalized type semigroup}, which is a further quotient of the object defined above, and on which $\precsim$ is a partial order. However, as we only need transitivity of $\precsim$, the given definition is sufficient for our purposes. 

We now recall the standard terminology of the theory.

\begin{definition}
A pair $(S, V)$, where $S\finsub \Gamma$ and $V \subset X$ is open is called a \emph{tower} if the sets $\{sV\}_{s \in S}$ (called \emph{levels} of the tower) are pairwise disjoint. We say that $S$ is the \emph{shape} of the tower and $V$ is the \emph{base}. A finite collection $\{(S_i, V_i)\}_{i=1}^n$ of towers is called a \emph{castle} if all the distinct levels $sV_i, s \in S_i, i=1,\ldots, n$ are pairwise disjoint.
\end{definition}

\begin{definition}
Fix a metric on $X$ compatible with the topology. An action $\alpha$ is \emph{almost finite} if for all $F\finsub\Gamma$ and all $\eps>0$ there exists a castle $\{(S_i,V_i)\}_{i=1}^n$ such that 
\begin{enumerate}
    \item $S_i\finsub\Gamma$ is $(F,\eps)$-invariant (that is, $|S_i\triangle FS_i|<\eps|S_i|$) for all $i$,
    \item $\mathrm{diam}(sV_i) < \eps$ for every $i=1, \ldots, n$ and every $g \in S_i$,
    \item $X \setminus \bigsqcup_{i=1}^nS_iV_i \precsim_\Gamma \sum_{i=1}^n \lfloor\eps\abs{S_i}\rfloor [V_i]$.\footnote{Technically a slight abuse of notation. However, elements in the type semigroup should be thought of as multisets, and can thus be (dynamically) compared to any set.}
\end{enumerate}
\end{definition}

Closely related to almost finiteness is the notion of (dynamical) comparison, of which we give several useful variations.

\begin{definition}
\label{def:comp}
    We say that an action has \emph{comparison} if for any open sets $U,V\subseteq X$ we have
    \[
        \mu(U)<\mu(V)\ \ \forall\mu\in M^{\Gamma}(X)\implies U\precsim_{\Gamma} V.
    \]
    Similarly, an action has \emph{comparison on multisets} if for any elements $\sum_{i=1}^n[U_i],\sum_{j=1}^m[V_j]$ in the type semigroup we have
    \[
        \sum_{i=1}^n\mu(U_i)<\sum_{j=1}^m\mu(V_j)\ \ \forall\mu\in M^{\Gamma}(X)\implies\sum_{i=1}^n[U_i]\precsim_{\Gamma}\sum_{j=1}^m[V_j].
    \]
    Finally, the action has $m$-comparison if for any open sets $U,V\subseteq X$ we have 
     \[
        \mu(U)<\mu(V)\ \ \forall\mu\in M^{\Gamma}(X)\implies [U]\precsim_{\Gamma} m[V].
    \]
\end{definition}

We refer to \cite{Ma21} for a more detailed treatment of these notions, as well as a connection with the algebraic structure of the (generalized) type semigroup.

\begin{definition}
An action has the \emph{small boundary property} (SBP) if there exists a basis $\cU$ for the topology on $X$ such that
\[
\mu(\partial U) = 0 \ \ \forall \mu\in M^{\Gamma}(X), U \in \cU.
\]
\end{definition}

Although seemingly unrelated, a surprising connection between the SBP and almost finiteness was uncovered in \cite{KerSza20}, where it was shown that the following are equivalent for free actions of discrete amenable groups:
\begin{enumerate}
    \item almost finiteness,
    \item SBP and comparison, and
    \item SBP and comparison on multisets.\footnote{the last condition does not appear in the original work \cite{KerSza20} but is easily seen to be equivalent with the same proof.}
\end{enumerate}
We remark that (i) implies (ii) for non-free actions as well, but the converse fails. Note also that the SBP is automatically satisfied if $X$ is zero-dimensional.

\subsection{Micro-supported actions}
\label{ssec: microsup}

\begin{definition}\label{def:microsup}
    A faithful action of $\Gamma$ on a topological space $Y$ is \emph{micro-supported} if the rigid stabilizer $\Gamma_U\coloneqq\{s\in\Gamma:y\in Y \setminus U\implies sy=y\}$ is non-trivial for any open set $U\subseteq Y$.
\end{definition}

Most groups of dynamical origin are equipped with a natural action on some topological space (usually, a Cantor set). As it turns out, for many interesting classes of groups this action is in fact micro-supported.

\begin{example}
The following groups admit faithful micro-supported actions on a Hausdorff Baire space with no isolated points:
\begin{enumerate}
    \item weakly branch groups,
    \item groups $\Gamma$ such that $\Alt(\mathfrak{G})\leq \Gamma\leq\Full(\mathfrak{G})$, where $\mathfrak{G}$ is an essentially principal minimal {\'e}tale groupoid with an infinite totally disconnected unit space.
\end{enumerate}
We refer to \cite{Nek22} for the relevant definitions and more on micro-supported actions and their applications.
\end{example}

Having such an action implies the existence of subgroups with certain properties, as the next lemma shows.

\begin{lemma}\label{lem:MicroSupAbelianEmbed}
    Assume that $\Gamma$ admits a faithful micro-supported action on a Hausdorff Baire space $Y$ with no isolated points. Then for any $F\finsub\Gamma$ and any $N\in\Nb$ there exist embeddings $H_1,H_2\emb\Gamma$ of finitely generated abelian groups with $|H_1|\geq N$ and $sH_1s^{-1}\subseteq H_2$ for all $s\in F$.
\end{lemma}

\begin{proof}
    Let $e\in F\finsub\Gamma$ and $N\in\Nb$. Let $s \in \Gamma$, let $\Fix(s)$ denote the set of its fixed points and consider the set $Y \setminus \partial(\Fix(s))$. It is clearly open and dense, hence the set
    \[
    Y_0 \coloneqq \bigcap_{s \in \Gamma} Y \setminus \partial(\Fix(s))
    \]
    is $G_\delta$ and dense, in particular non-empty. Note that if some element of $\Gamma$ fixes a point in $Y_0$, then it pointwise fixes a neighbourhood of it.

    Let $y_0\in Y_0$ and write $Fy_0=\{y_0,\dots,y_n\}$. For $i\in\{0,\dots,n\}$, define $F_i\coloneqq\{s\in F:sy_0=y_i\}$. Then the finite set $\cup_{i=0}^nF_i^{-1}F_i$ fixes $y_0$, and therefore it fixes an open neighbourhood $U$ of it. Next, we find a sufficiently small open set $V\subset U$ so that the sets $sV$ and $tV$ are pairwise disjoint whenever $s \in F_i$ and $t \in F_j$ for $i \ne j$. It follows that for any $s, t \in F$ either $sV\cap tV=\emptyset$ or $s|_V=t|_V$.
    
    Using Hausdorffness and the lack of isolated points, $V$ is infinite and we can pick $N$ non-empty disjoint open subsets $V_1,\dots,V_N$ of it. Since the action of $\Gamma$ on $Y$ is micro-supported, there exist non-trivial elements $t_j\in\Gamma_{V_j}$ for all $j\in\{1,\dots,N\}$. By construction, for any $i, j\in\{1,\dots,N\}$ and any $s, t \in F$, the elements $st_is^{-1}$ and $tt_jt^{-1}$ either coincide or have disjoint support (and hence commute). Thus, the groups $H_1\coloneqq\langle t_1,\dots,t_N\rangle$ and $H_2\coloneqq\langle\{st_js^{-1}:s\in F,j\in\{1,\dots N\}\}\rangle$ are both abelian which finishes the proof.
\end{proof}

\begin{remark}\label{rem: microsup weakening}
It can be seen from the proof that we may substitute the condition of the action being micro-supported with the following weaker property: for any $F\finsub\Gamma$ and any $N\in\Nb$ there exists an open set $U \subset Y$ such that $\Gamma_U$ contains an abelian subgroup of order at least $N$ and for any pair $s, t \in F$ either $\left.s\right|_U = \left.t\right|_U$ or the sets $sU$ and $tU$ are disjoint. 

This allows us to show that the conclusion of Lemma \ref{lem:MicroSupAbelianEmbed} holds under slightly more general circumstances. For instance, this is the case for groups $\Gamma$ such that $\Alt(\mathfrak{G})\leq \Gamma\leq\Full(\mathfrak{G})$, where $\mathfrak{G}$ is an {\'e}tale groupoid with an infinite totally disconnected unit space and either
\begin{itemize}
    \item there is a point $x\in\mathfrak{G}^{(0)}$ such that $\mathfrak{G}_{(x)}$ is trivial and $\mathfrak{G}x$ is infinite or
    \item the set of points $x\in\mathfrak{G}^{(0)}$ such that the orbit $\mathfrak{G}x$ is infinite is not meager.
\end{itemize}
In other words, the conditions of being essentially principal and minimal can be significantly relaxed.
\end{remark}

\section{Comparison from good subgroups and the main results}
\label{sec: Comp_from_subgrps}

\begin{definition}
\label{def:good_Subgroups}
Let $\Gamma \act X$ be an action of a countable discrete group on a compact metrizable space. We say that the action has \emph{good subgroups} if for any finite set $F \subset \Gamma$ and $N\in\mathbb{N}$ there are subgroups $\Lambda_1 \le \Lambda_2 \le \Gamma$ of cardinality at least $N$ such that 
\begin{enumerate}
    \item the action $\Lambda_1 \act X$ is almost finite,
    \item the action $\Lambda_2 \act X$ has comparison on multisets, and
    \item $f\Lambda_1f^{-1} \subset \Lambda_2$ for any $f \in F$.
\end{enumerate}
\end{definition}

\begin{remark}\label{rem: SBP extension}
Note that since $\Lambda_1 \act X$ is almost finite, it has the small boundary property, which passes to $\Gamma \act X$. This also holds in the case when $\Lambda_1$ is finite -- indeed, a free action of a finite group is almost finite if and only if the space $X$ is totally disconnected. In such a case any other action $\Gamma \act X$ trivially has the small boundary property. 
\end{remark}

We recall two lemmas which will be useful later. The first is a slight generalization of \cite[Lemma 3.3]{Ker20} and is proved in the exact same way. The second is immediate from the definition of the SBP.

\begin{lemma}\label{lem: eps-lemma}
Fix a compatible metric $\rho$ on $X$. Let $C \subset X$ be closed, let $B \subset X$ be open, and let $\Omega$ be a weak$^*$ closed subset of $M(X)$. Let $F_1, F_2$ be finite subsets of $\Gamma$. Suppose that for some $\delta
\in \mathbb{R}$ we have that
\[
\frac{1}{|F_1|}\sum_{g\in F_1}\mu(gC) + \delta < \frac{1}{|F_2|}\sum_{g\in F_2}\mu(gB)
\]
for every $\mu \in \Omega$. Then there is some $\eps > 0$ such that the sets
\[
C^\eps = \{x \in X \colon \rho(x, C) < \eps\} \quad \mbox{and} \quad B^{-\eps} = \{x \in X \colon \rho(x, X \setminus B) > \eps\}
\]
satisfy
\[
\frac{1}{|F_1|}\sum_{g\in F_1}\mu(gC^\eps) + \delta + \eps < \frac{1}{|F_2|}\sum_{g\in F_2}\mu(gB^{-\eps}).
\]
\end{lemma}

\begin{lemma}\label{lem: SBP}
Suppose that $\Gamma \act X$ has the small boundary property and let $C \subset X$ be closed. Then for any $\eps > 0$ there is an open set $U$ such that $C \subset U \subset C^\eps$ and
\[
\mu(\partial U) = 0
\]
for any $\mu \in M^\Gamma(X)$.
\end{lemma}

The following lemma is a generalization of \cite[Theorem A]{Nar23}, and the proof follows along the same lines.

\begin{lemma}
\label{ComparisonFromSubgroups}
Let $\Gamma \act X$ be an action with good subgroups, let $C \subset X$ be closed, and let $D \subset X$ be open. Suppose that there exist finite sets $F_1, F_2 \subset G$ such that for every $x \in X$
\begin{equation}
\label{AverageMeasureIneq}
\frac{1}{\abs{F_1}}\sum_{g\in F_1}\delta_x(gC)<\frac{1}{\abs{F_2}}\sum_{g\in F_2}\delta_x(gD),
\end{equation}
where $\delta_x$ is the Dirac measure at $x$. Then $C \precsim D$.
\end{lemma}

\begin{proof}
First of all, given any $\mu \in M(X)$, integrating \eqref{AverageMeasureIneq} yields
\[
\frac{1}{\abs{F_1}}\sum_{g\in F_1}\mu(gC)<\frac{1}{\abs{F_2}}\sum_{g\in F_2}\mu(gD).
\]
Applying Lemma \ref{lem: eps-lemma} shows that there exists some $\gamma > 0$ such that
\[
\frac{1}{\abs{F_1}}\sum_{g\in F_1}\mu(gC) + \gamma < \frac{1}{\abs{F_2}}\sum_{g\in F_2}\mu(gD).
\]
Set $n_i = |F_i|$ for $i=1,2$, and let $n=\max\{n_1,n_2\}$. Let $F =F_1\cup F_2\cup\{e\}$. Since $\Gamma\act X$ has good subgroups, we can find $\Lambda_1,\Lambda_2<\Gamma$, with $\Lambda_1\act X$ being almost finite and $\Lambda_2\act X$ having comparison on multisets, of sufficiently large cardinality such that
\[
\bigcup_{g\in F}g\Lambda_1g^{-1}\subseteq \Lambda_2,
\]
and a castle $(S_j, U_j)_{j=1}^m$ with $S_j\subseteq \Lambda_1$ such that $\min_{j=1,\dots,m}|S_j|>\frac{4n}{\gamma}$
and
\[
\mu\left(X\setminus\bigsqcup_{j=1}^mS_jU_j\right)<\frac{\gamma}{4}\ \forall\mu\in M^{\Lambda_1}(X).
\]

Since $\Lambda_1\act X$ is almost finite it has the small boundary property. Using Lemmas \ref{lem: eps-lemma} and \ref{lem: SBP}, we thus can find 
\begin{itemize}
    \item an open neighbourhood $A$ of $C$,
    \item an open subset $B \subset D$, and
    \item open subsets $V_j \subset U_j$ for $j = 1, 2, \ldots, m$
\end{itemize}
such that 
\begin{equation}\label{AvgMeasureIneq2}
\begin{aligned}
    \mu(\partial A) = \mu(\partial B) = \mu(\partial V_j) &= 0 \ \forall \mu \in M^{\Lambda_1}(X), j = 1, 2, \ldots, m \\
    \frac{1}{|F_1|}\sum_{g\in F_1}\mu(gA)+\gamma&<\frac{1}{|F_2|}\sum_{g\in F_2}\mu(gB)\ \forall\mu\in M(X) \\
    \mu\left(X\setminus\bigsqcup_{j=1}^mS_jV_j\right)&<\frac{\gamma}{4}\ \forall\mu\in M^{\Lambda_1}(X).
\end{aligned}
\end{equation}

Next, we can assume that all the levels of the castle $(S_j, V_j)_{j=1}^m$ are either contained in or disjoint from $A$ and $B$ by breaking the towers according to the intersection pattern\footnote{Take in each step the first tower that contains a level intersecting both $A$ and its complement, and the first such level, say $sV$. Replace the base of the tower with two new bases $V\cap s^{-1}A$ and $V\setminus\overline{s^{-1}A}$, and replace the whole tower by the two towers generated by the two new bases using the same shape. This process will end after a finite amount of steps, and in each step we are only throwing away parts of the boundaries. Repeat for $B$.} without changing the measure of the castle under any $\Lambda_1$-invariant measure. Denote by $a_j$, $b_j$ ($j=1,\dots,m$) the number of $V_j$-levels contained in $A$ and $B$, respectively.

Denoting by $R^\epsilon$ an $\epsilon$-neighbourhood of the remainder $R=X\setminus\bigsqcup_{j=1}^mS_jV_j$ such that $\mu(R^\epsilon)<\frac{\gamma}{4}$ for all $\mu\in M^{\Lambda_1}(X)$, we have
\[
A\precsim_\Gamma\sum_{i=1}^m\sum_{g\in F_1}\ceil{\frac{a_i}{n_1}}[gV_i]+[R^\epsilon].
\]
Moreover,
\[
\nu\left(\sum_{j=1}^m\sum_{g\in F_1}\ceil{\frac{a_j}{n_1}}[gV_j]+[R^\epsilon]\right)-\frac{1}{n_1}\sum_{g\in F_1}\nu(gA)\overset{(*)}{\le}\sum_{j=1}^m\sum_{g\in F_1}\left(\ceil{\frac{a_j}{n_1}}-\frac{a_j}{n_1}\right)\nu(gV_j)+\nu(R^\epsilon)
\]
for every $\nu\in M^{\Lambda_2}(X)$, where for $(*)$ we have used that for every $j=1,\dots,m$ and every $g\in F$ we have $gS_jV_j=\tilde{S_j}gV_j$, $\tilde{S_j}=gS_jg^{-1}\subseteq \Lambda_2$ by our choice of $\Lambda_2$. Furthermore, for all $g\in F_1$ and $\nu\in M^{\Lambda_2}(X)$ we have
\[
\sum_{j=1}^m\nu(gV_j)=\sum_{j=1}^m\frac{1}{|\tilde{S_j}|}\sum_{s\in\tilde{S_j}}\nu(sgV_j)\leq\frac{1}{\min|\tilde{S_j}|}<\frac{\gamma}{4n}
\]
and thus
\[
\nu\left(\sum_{j=1}^m\sum_{g\in F_1}\ceil{\frac{a_j}{n_1}}[gV_j]+[R^\epsilon]\right)-\frac{1}{n_1}\sum_{g\in F_1}\nu(gA)<\frac{\gamma}{2}.
\]
Similarly, we get
\[
\sum_{j=1}^m\sum_{g\in F_2}\floor{\frac{b_j}{n_2}}[gV_j]\precsim_\Gamma B
\]
and
\[
\frac{1}{n_2}\sum_{g\in F_2}\nu(gB)-\nu\left(\sum_{j=1}^m\sum_{g\in F_2}\floor{\frac{b_j}{n_2}}[gV_j]\right)<\frac{\gamma}{2}.
\]
Thus, by $\eqref{AvgMeasureIneq2}$ we obtain
\[
\nu\left(\sum_{j=1}^m\sum_{g\in F_1}\ceil{\frac{a_j}{n_1}}[gV_j] + [R^\epsilon]\right)<\nu\left(\sum_{j=1}^m\sum_{g\in F_2}\floor{\frac{b_j}{n_2}}[gV_j]\right)
\]
for all $\nu\in M^{\Lambda_2}(X)$ and so, by comparison on multisets for $\Lambda_2$, we have
\[
\sum_{j=1}^m\sum_{g\in F_1}\ceil{\frac{a_j}{n_1}}[gV_j] + [R^\epsilon]\precsim_{\Lambda_2}\sum_{j=1}^m\sum_{g\in F_2}\floor{\frac{b_j}{n_2}}[gV_j],
\]
which implies
\[
C \subset A\precsim_\Gamma B \subset D.
\]
\end{proof}

We are now ready to prove the main theorem.

\begin{theorem}\label{thm:amen_comp_from_g_s}
Let $\Gamma$ be amenable and suppose that an action $\Gamma \act X$ has good subgroups. Then $\Gamma \act X$ has comparison. In addition, if $\Gamma \act X$ is free then it is almost finite.
\end{theorem}

\begin{proof}
    Comparison is immediate from Lemma \ref{ComparisonFromSubgroups}, since for any $C,D$ with
    \[
    \mu(C)<\mu(D)\ \forall\mu\in M^\Gamma(X)
    \]
    we can achieve \eqref{AverageMeasureIneq} by choosing $F_1=F_2$ to be a sufficiently large F{\o}lner set. The action has SBP by Remark \ref{rem: SBP extension} and therefore is almost finite if free.
\end{proof}

\begin{corollary}
\label{cor: amenTFG_AF}
Suppose that $\Gamma$ admits a faithful micro-supported action on a Hausdorff Baire space with no isolated points. Then any free action of $\Gamma$ on a finite-dimensional space $X$ is almost finite.
\end{corollary}

\begin{proof}
It is known that free actions of abelian groups on zero-dimensional spaces are almost finite and therefore have comparison (on multisets). Thus, Lemma \ref{lem:MicroSupAbelianEmbed} shows that any free action of $\Gamma$ on a zero-dimensional space has good subgroups. Theorem~\ref{thm:amen_comp_from_g_s} guarantees that these actions have comparison, and thus are almost finite. Applying \cite[Theorem B]{KerSza20} then finishes the proof.
\end{proof}

\begin{remark}
    The above gives another potential way to tackle the long-standing question of amenability for Thompson's group $F$ and for the IET group---they are amenable if and only if all their free actions on finite-dimensional spaces are almost finite.
\end{remark}

\end{document}